\documentclass[12pt]{amsart}
\usepackage{amsmath,amssymb,amsfonts,amsthm,amsopn}
\usepackage{graphics}
\usepackage{xcolor}
\usepackage{pb-diagram}
\usepackage[title]{appendix}
\usepackage{tikz-cd}
\newcommand{\tf}{time-frequency}

\newcommand{\tfs}{time-frequency shift}

\newtheorem{theorem}{Theorem}[section]
\newtheorem{lemma}[theorem]{Lemma}
\newtheorem{corollary}[theorem]{Corollary}
\newtheorem{proposition}[theorem]{Proposition}
\newtheorem{definition}[theorem]{Definition}

\newtheorem{example}[theorem]{Example}
\newtheorem{remark}[theorem]{Remark}

\newcommand{\beqa}{\begin{eqnarray*}}
\newcommand{\eeqa}{\end{eqnarray*}}

\DeclareMathOperator*{\Lift}{Lift}

\newcommand{\field}[1]{\mathbb{#1}}
\newcommand{\bR}{\field{R}}        %  real numbers
\newcommand{\bC}{\field{C}}        %  complex numbers
\def\la{\lambda}

\def\cF{\mathcal{F}}              % Calligraphic Letters
\def\cS{\mathcal{S}}
\def\cD{\mathcal{D}}

\def\cE{\mathcal{E}}

\def\cM{\mathcal{M}}
\def\cK{\mathcal{K}}

\def\cA{\mathcal{A}}

\def\rd{\bR^d}

\def\rdd{{\bR^{2d}}}

\def\lrd{L^2(\rd)}

\def\R{\right)}

\def\<{\left<}
\def\>{\right>}

\def\mv1{M_v^1}

\def\mpq{M^{p,q}}
\def\Mmpq{M_m^{p,q}}
\def\phas{(x,\xi )}
\def\mn{(m,n)}
\def\mn'{(m',n')}

\newcommand{\norm}[1]{\lVert#1\rVert}

\def\R{\mathbb{R}}
\def\Ren{\mathbb{R}^d}

\def\sch{\mathcal{S}}

\def\Sn2{S_{2}(L^{2}(\Ren))}
\def\S1{S_{1}(L^{2}(\Ren))}
\def\sig00{\sigma_{0,0}}
\def\la{\langle}
\def\ra{\rangle}

\begin{document}
\begin{abstract} 
We provide a comprehensive overview of the theoretical framework surrounding modulation spaces and their characterizations, particularly focusing on the role of metaplectic operators and time-frequency representations.
We highlight the metaplectic action which is hidden in their construction and guarantees equivalent (quasi-)norms for such spaces.  In particular, this work provides new characterizations via the submanifold of shift-invertible symplectic matrices. Similar results hold for the Wiener amalgam spaces.
\end{abstract}

\title[Excursus on modulation spaces]{Excursus on modulation spaces via metaplectic  operators and related time-frequency representations}

\author{Elena Cordero}
\address{Universit\`a di Torino, Dipartimento di Matematica, via Carlo Alberto 10, 10123 Torino, Italy}
\email{elena.cordero@unito.it}
\author{Gianluca Giacchi}
\address{Università di Bologna, Dipartimento di Matematica,  Piazza di Porta San Donato 5, 40126 Bologna, Italy; Institute of Systems Engineering, School of Engineering, HES-SO Valais-Wallis, Rue de l'Industrie 21, 1950 Sion, Switzerland; Lausanne University Hospital and University of Lausanne, Lausanne, Department of Diagnostic and Interventional Radiology, Rue du Bugnon 46, Lausanne 1011, Switzerland. The Sense Innovation and Research Center, Avenue de Provence 82
1007, Lausanne and Ch. de l’Agasse 5, 1950 Sion, Switzerland.}
\email{gianluca.giacchi2@unibo.it}
%\author{Luigi Rodino}
%\address{}
%\email{}

\thanks{}

%\subjclass{35S05,35S30,
%47G30, 42C15}
%\date{}
\maketitle

\section{Introduction}

This work is a detailed overview of the theory surrounding modulation spaces and their characterization using metaplectic operators and time-frequency representations. It covers the historical development of modulation spaces from their introduction by Feichtinger in 1983  to recent advancements in the field and provides new results in this framework.  \par
Here's a breakdown of the main points discussed:
\begin{enumerate}
	\item[] \textbf{Background:} We  introduce modulation spaces, which are fundamental in various disciplines including time-frequency analysis, PDEs, quantum mechanics, and signal analysis. 
	\item[] \textbf{Metaplectic Operators and Invariance Properties:} We discuss the importance of metaplectic operators and their invariance properties in characterizing modulation spaces, presenting results regarding the conditions under which metaplectic operators extend to isomorphisms on modulation spaces.
	\item[] \textbf{Time-Frequency Representations:} We present various time-frequency representations, including the short-time Fourier transform (STFT) and the Wigner distribution, which arise as images of metaplectic operators and are called  metaplectic Wigner distributions. We highlight their role in characterizing modulation spaces and provide a reproducing formula for these representations.
	\item[] \textbf{Characterization of Modulation Spaces:} 
	We  exhibit characterizations of modulation spaces using  metaplectic Wigner distributions. These characterizations involve shift-invertible properties and block decompositions of symplectic matrices.
	\item[] \textbf{Independence of Symplectic Matrix Blocks:} We showcase a new result regarding the independence of certain blocks of symplectic matrices in the characterization of modulation spaces using metaplectic Wigner distributions.
\end{enumerate}

\vspace{0.3truecm}
Modulation and Wiener amalgam spaces  were introduced by Feichtinger in  \cite{F1} (see \cite{Feichtinger_1981_Banach,Feichtinger_1990_Generalized} for the Wiener case) in the framework of time-frequency analysis and extended to the quasi-Banach setting  by Galperin and Samarah \cite{Galperin2004} in 2004. Since the 2000s they have become popular in many different environments, ranging from PDEs to quantum mechanics, pseudo-differential theory and signal analysis. Hundreds of papers use these spaces as a natural framework. Since it is impossible to cite them all, we simply refer to the textbooks  \cite{KB2020,Elena-book,book}. 

Recently, a series of contributions have highlighted the importance of metaplectic operators and related time-frequency representations in the (quasi-)norm characterization of these spaces \cite{CNT2019,CGshiftinvertible,CGR2022,CR2021,CR2022,Fuhr,Gos11,grohs2022,ZJQ21,Zhang21bis,Zhang21}, providing applications in signal analysis, phase retrieval, and PDE's.
 
A key point in this study, as observed in \cite{Fuhr}, is  the \emph{invariance properties} of metaplectic operators, asking under which conditions a metaplectic operator  $\hat{A}$, initially defined as a unitary operator on $\lrd$, extends to an isomorphism on the modulation spaces $\Mmpq(\rd)$.
Recall that the metaplectic group $Mp(d,\bR)$ is the two-fold cover of the symplectic group $Sp(d,\bR)$, i.e., there exists a surjective Lie group homomorphism $\pi^{Mp}: Mp(d,\bR)\to Sp(d,\bR)$. 

The translation and modulation operators are defined by		$$
T_xf(t)=f(t-x),\quad M_{\xi}f(t)= e^{2\pi i \xi
	t}f(t), \, t,x,\xi\in\rd.
$$
Their composition yields the so-called \emph{time-frequency shift}
$$
\pi(z)f=\pi\phas f=M_\xi T_xf(t)=e^{2\pi i \xi
	t}f(t-x),\quad z=\phas\in\rdd.
$$
Modulation spaces $M^{p,q}_m(\rd)$ are classically defined in terms of the short-time Fourier transform (STFT), i.e.,
\begin{equation}\label{STFTp}
	V_gf(x,\xi)=\la f,\pi\phas g\ra=\int_{\rd}f(t)\bar g(t-x)e^{-2\pi i\xi t}dt, \qquad f\in L^2(\rd), \ x,\xi\in\rd, 
\end{equation}
where $g\in \cS(\rd)\setminus\{0\}$ is the so-called window function. The definition can be extended to $(f,g)\in\cS'(\rd)\times\cS(\rd)$, cf. Section \ref{subsec:23} for details.

 Modulation spaces $M^{p,q}_m(\rd)$ are subclasses of tempered distributions $f\in\cS'(\rd)$, with
\[
f\in M^{p,q}_m(\rd) \quad \Longleftrightarrow \quad V_gf\in L^{p,q}_m(\rdd). 
\]

Recently,  F\"uhr and  Shafkulovska  \cite[Theorem 3.2]{Fuhr} (see \cite{CGshiftinvertible} for the quasi-Banach setting)  proved the following:
\begin{theorem}\label{thmFuhr} Consider $0<p,q\leq\infty$, $\hat A\in Mp(d,\bR)$, $\pi^{Mp}(\hat{A})=A\in Sp(d,\bR)$. The following statements are equivalent:\\
	(i) $\hat{A}: \mpq\to \mpq$ is well defined;\\
	(ii) $\hat{A}: \mpq\to \mpq$ is well defined and bounded (in fact, it is an isomorphism);\\
	(iii) One out of the two conditions below holds:\\
	\hspace*{0.3cm} \,\,\,\,(iii.1) $p=q$,\\
	\hspace*{0.5truecm} (iii.2) $p\not=q$ and $A$ is an upper triangular matrix.\\
	\end{theorem}
	This result can be lifted to weighted modulation spaces $\Mmpq$ if the weight $m$ satisfies $m\asymp m\circ A^{-1}$.

	From \cite[Lemma 9.4.3]{book}  we infer the action of the STFT on the metaplectic operator $\hat{A}$ such that $\pi^{Mp}(\hat{A})=A$:
	\begin{equation}\label{e1}
	|V_g (\hat{A}f)\phas|=	|V_{\hat{A}^{-1}g} f(A^{-1}\phas)|.
	\end{equation}
 
The equality above shows the interaction between $\hat{A}: \mpq(\rd) \to \mpq(\rd)$ and $D_A: L^{p,q}(\rdd)\to L^{p,q}(\rdd)$, where $D_A F(z):=F(A^{-1}z)$, highlighted by the diagram below. 

%%%%% !!!!!
\begin{equation}\label{diagramma}
\begin{tikzcd}
	M^{p,q}(\rd) \arrow[rr,"V_g"] \arrow[d,"\hat A"] & 	& V_g(M^{p,q}(\rd)) \arrow[rr,hook] \arrow[d,dashrightarrow] &	& L^{p,q}(\rdd) \arrow[d,"D_A"]\\ 
	M^{p,q}(\rd) \arrow[rr,"V_{\hat A^{-1}g}"] & 		& V_{\hat A^{-1}g}(M^{p,q}(\rd)) \arrow[rr,hook] 	&	& L^{p,q}(\rdd)
\end{tikzcd}
\end{equation}

Then, Theorem \ref{thmFuhr} is a direct consequence of the result below \cite[Theorem 3.3]{Fuhr}.

\begin{theorem}\label{e3}
	For $0<p,q\leq\infty$, $A\in Sp(d,\bR)$, the following statements are equivalent:\\
	(i) $D_A: L^{p,q}(\rdd)\to L^{p,q}(\rdd)$ is everywhere defined.\\
	(ii) $D_A: L^{p,q}(\rdd)\to L^{p,q}(\rdd)$ is everywhere defined and bounded.\\
	(iii) $D_A: V_{\hat{A}^{-1}g}({M}^{p,q})\to L^{p,q}(\rdd)$ is everywhere defined.\\
	(iv) $D_A: V_{\hat{A}^{-1}g}({M}^{p,q})\to L^{p,q}(\rdd)$ is everywhere defined and bounded.\\
	(v)  One out of the two conditions below holds:\\
	\hspace*{0.3cm} \,\,\,\,(v.1) $p=q$, \\ 
	\hspace*{0.5truecm} (v.2) $p\not=q$ and $A$ is an upper triangular matrix.\\
\end{theorem}

A clarifying example, provided by \cite{Fuhr}, is the following:
for $p\not=q$, consider two radially symmetric functions $f\in L^p(\rd)\setminus L^q(\rd)$ and $g\in L^q(\rd)\setminus L^p(\rd)$, and consider the standard symplectic matrix:
\begin{equation}\label{defJ}
	J=\begin{pmatrix} 0_{d\times d} & I_{d\times d} \\ -I_{d\times d} & 0_{d\times d} \end{pmatrix},
\end{equation}
where $I_{d\times d}\in\bR^{d\times d}$ is the identity matrix, whereas $0_{d\times d}$ is the matrix of $\bR^{d\times d}$ having all zero entries. The tensor product $(f\otimes g)\phas=f(x)g(\xi)$ is in $L^{p,q}(\rdd)$, but $(g\otimes f)\phas=D_J(f\otimes g)\phas\notin L^{p,q}(\rdd)$, hence $D_J: L^{p,q}(\rdd)\to L^{p,q}(\rdd)$ is not well defined.

The proof of Theorem \ref{e3} for $p\not=q$ lies on the following issue:
\begin{theorem}\label{e4}
	If $S\in GL(d,\bR)$ is an upper triangular matrix and $0<p,q\leq\infty$, then
	$$D_S: L^{p,q}(\rdd)\to L^{p,q}(\rdd)$$
	is, up to a constant $C_S>0$, a norm-preserving isomorphism with bounded inverse $D_S^{-1}=D_ {S^{-1}}$.
\end{theorem}
Notice that the matrix $S$ does not need to be symplectic, but only invertible.

These results have inspired new characterizations of modulation spaces via time-frequency distributions built by using metaplectic operators: the so-called  metaplectic Wigner distributions, defined as follows.

For $\hat\cA\in Mp(2d,\bR)$, with $\pi^{Mp}(\hat{\cA})=\cA\in Sp(2d,\bR)$ (notice that the number of variables is doubled!) we call the metaplectic Wigner distribution $	W_\cA(f,g)$ the time-frequency representation
\begin{equation}\label{WApre}
	W_\cA(f,g)=\hat\cA(f\otimes \bar g),\quad f,g\in\lrd.
\end{equation}

Examples are the STFT, the (cross-)$\tau$-Wigner distribution:
\begin{equation}\label{tauWigner}
W_\tau(f,g)(x,\xi)=\int_{\rd}f(x+\tau t)\bar g(x-(1-\tau)t) e^{-2\pi i\xi t}dt, \qquad x,\xi\in\rd;
\end{equation}
 in particular, the Wigner transform,  defined  as
\begin{equation}\label{Wignerp}
	W(f,g)(x,\xi)=\int_{\rd}f\left(x+\frac{t}{2}\right)\bar g\left(x-\frac{t}{2}\right) e^{-2\pi i\xi t}dt, \quad f,g\in\lrd,\quad x,\xi\in\rd.
\end{equation}
For their related metaplectic operators $\hat{\cA}$ and symplectic matrix $\pi^{Mp}(\hat{\cA})=\cA$ we refer to the following section.

Similarly to the STFT, these time-frequency representations enjoy a \emph{reproducing formula}, cf. \cite[Lemma 3.6]{CGshiftinvertible}:
\begin{lemma}\label{intertFormula}
Consider $\hat\cA\in Mp(2d,\bR)$, with $\pi^{Mp}(\hat{\cA})=\cA\in Sp(2d,\bR)$, $\gamma,g\in\cS(\rd)$  such that $\langle \gamma,g\rangle\neq0$ and $f\in\cS'(\rd)$. Then,
	\begin{equation}\label{e6}
		W_\cA(f,g)=\frac{1}{\langle\gamma, g\rangle}\int_{\rdd}V_gf(w)W_\cA(\pi(w)\gamma,g) dw,
	\end{equation}
	with equality in $\cS'(\rdd)$, the integral being intended in the weak sense.
\end{lemma}
From the right-hand side we infer that the key point becomes the action of $W_\cA$ on the time-frequency shift $\pi(w)$, which can be computed explicitly. 
For, assume that $\cA$ has got the block decomposition
 \begin{equation}\label{blockA}
	\cA=\begin{pmatrix}
		A_{11} & A_{12} & A_{13} & A_{14}\\
		A_{21} & A_{22} & A_{23} & A_{24}\\
		A_{31} & A_{32} & A_{33} & A_{34}\\
		A_{41} & A_{42} & A_{43} & A_{44}
	\end{pmatrix}
\end{equation}
and consider its sub-matrix 
\begin{equation}\label{defEA}
	E_\cA=\begin{pmatrix}
		A_{11} & A_{13}\\
		A_{21} & A_{23}
	\end{pmatrix},
\end{equation} 
then
\begin{equation}
			|W_\cA(\pi(w)f,g)(z)|=|W_\cA(f,g)(z-{E_\cA }w)|,\quad f,g\in\lrd.
\end{equation}

The equality above suggests the following definition.
\begin{definition}
	Under the notation above, we say that $W_\cA$ (or, by abuse, $\cA$) is \textbf{shift-invertible} if $E_\cA\in GL(2d,\bR)$.
\end{definition}
The representation formula \eqref{e6} and the invertibility property enjoyed by $E_\cA$ (which plays the role of the matrix $S$ in Theorem \ref{e4})	 are the main ingredients for the characterization of modulation spaces via metaplectic Wigner distributions, proved in Theorem 3.7 of \cite{CGshiftinvertible}:
\begin{theorem} \label{thmF}
	Let $W_\cA$ be shift-invertible with $E_\cA$ upper triangular.  Fix a non-zero window function $g\in \cS(\rd)$. If $m\asymp m\circ E_\cA^{-1}$, $1\leq p,q\leq\infty$,  
	\begin{equation}\label{charmod}
		f\in M^{p,q}_m(\rd) \qquad \Leftrightarrow \qquad W_\cA(f,g)\in L^{p,q}_m(\rdd),
	\end{equation}
	with equivalence of norms.
	
	If $p=q$ the matrix $E_\cA$ does not need to be upper triangular.
\end{theorem}
For the quasi-Banach modulation spaces, the characterization above still holds true, but the proof is obtained by different methods. 
The main issue is to understand  which symplectic matrices give a sub-matrix $E_\cA$ invertible, so that the corresponding distributions $W_\cA$ can characterize modulation spaces.

First, let us introduce the chirp function related to the symmetric matrix $C\in\bR^{2d\times 2d}$, defined as
$$\Phi_C(t)=e^{\pi i t\cdot Ct},\quad t\in\rd$$
and, for $E\in GL(d,\bR)$, the dilation operator $\mathfrak{T}_E$ (which is a metaplectic operator, see below)
 $$\mathfrak{T}_E:=|\det(E)|^{1/2}\,f(E\,\cdot).$$
 The answer to our problem is contained in the recent contribution \cite{CGframes} and can be summarized as follows:
\begin{theorem}\label{GiuLaMaschera}
	Let $W_\cA$ be a shift-invertible metaplectic Wigner distribution with $\cA\in Sp(2d,\bR)$. Then the symplectic matrix $\cA$ can be split into the product of four symplectic matrices:
	\[
	\cA=\cD_{E_\cA^{-1}}V_{M_\cA+L}V_L^T \Lift(G_\cA),
	\]
see Section $3$ for their definition. The corresponding $W_\cA$ becomes
$$	W_\cA(f,g)=\mathfrak{T}_{E_\cA^{-1}}\Phi_{M_\cA+L}V_{\widehat{\delta_\cA} g}f,\quad f,g\in L^2(\rd).$$
$\widehat{\delta_\cA} \in Mp(d,\bR)$ is called \emph{deformation operator} and can be explicitly computed, see Section $3$ for details.
\end{theorem}
Roughly speaking, the result above says that 
\vspace{0.5truecm}

\centerline{\emph{\bf $W_\cA$ is shift-invertible if and only if $W_\cA$ is a STFT}}
\centerline{ \emph{\bf  up to linear change of variables and products-by-chirps}.}
\vspace{0.5truecm}
The invariance of $L^{p,q}(\rdd)$ spaces under dilations by upper-triangular matrices stated in Theorem \ref{e4} suggests that any dilation of this type, applied to the STFT, give equivalent modulation norms. 

The characterization of Theorem \ref{GiuLaMaschera} clearly depends on the blocks of $\cA$, does not clarify whether  $E_\cA$, $M_\cA$ and $G_\cA$ can be chosen independently of each other. The main result of this work expresses their independence:

\begin{theorem}\label{thmG}
For every $E\in GL(2d,\bR)$, $C\in Sym(2d,\bR)$ and $\delta\in Mp(d,\bR)$, the metaplectic Wigner distribution:
\begin{equation}\label{exprWAnew}
	W_\cA(f,g)(z)=|\det(E)|^{-1}\Phi_{C}(E^{-1}z) V_{\hat\delta g}f(E^{-1}z), \qquad f,g\in L^2(\rd), \ z\in\rdd
\end{equation}
is shift-invertible. Conversely, for every shift-invertible $W_\cA$ there exist $E\in GL(2d,\bR)$, $C\in\bR^{2d\times2d}$ symmetric and $\hat\delta\in Mp(d,\bR)$ such that (\ref{exprWAnew}) holds.
\end{theorem}
This issue is proven in Section $3$.

\section{Preliminaries}\label{sec:preliminaries}
\textbf{Notation.} We denote by %$t^2=t\cdot t$,  $t\in\rd$, and
$xy=x\cdot y$ the scalar product on $\Ren$.  The space   $\sch(\Ren)$ is the Schwartz class whereas its dual $\sch'(\Ren)$ is the space of temperate distributions. The brackets  $\la f,g\ra$ are the extension to $\sch' (\Ren)\times\sch (\Ren)$ of the inner product $\la f,g\ra=\int f(t){\overline {g(t)}}dt$ on $L^2(\Ren)$ (conjugate-linear in the second component). A point in the phase space (or \tf\ space) is written as
$z=(x,\xi)\in\rdd$, and  the corresponding phase-space shift (\tfs )
acts   as
$
\pi (z)f(t) = e^{2\pi i \xi t} f(t-x), \, \quad t\in\rd.
$

The notation $f\lesssim g$ means that there exists $C>0$ such that $ f(x)\leq Cg(x)$  for every $x$. The symbol $\lesssim_t$ is used to stress that $C=C(t)$. If $ g\lesssim f\lesssim g$ (equivalently, $ f \lesssim g\lesssim f$), we write $f\asymp g$. Given two measurable functions $f,g:\rd\to\bC$, we set $f\otimes g(x,y):=f(x)g(y)$. If $X(\rd)$ is any among $L^2(\rd),\cS(\rd),\cS'(\rd)$, $X\otimes X$ is the unique completion of $\text{span}\{x\otimes y : x\in X(\rd)\}$ with respect to the (usual) topology of $X(\rdd)$. Thus, the operator $f\otimes g\in\cS'(\rdd)$  characterized by its action on $\varphi\otimes\psi\in\cS(\rdd)$
\[
	\la f\otimes g,\varphi\otimes\psi\ra = \la f,\varphi\ra\la g,\psi\ra,\quad  \forall f,g\in\cS'(\rd),
\]
extends uniquely to a tempered distribution of $\cS'(\rdd)$. The subspace $\text{span}\{f\otimes g: f,g\in\cS'(\rd)\}$ is dense in $\cS'(\rdd)$.

$GL(d,\bR)$ stands for the group of $d\times d$ invertible matrices, whereas $Sym(d,\bR)=\{C\in\bR^{d\times d} \ : \ C \ is \ symmetric\}$.

%A set $\Lambda\subset \rd$ is called \emph{lattice} if there is a matrix $A\in GL(d,\bR)$  such that $\Lambda=A\zd =\{Az : z\in\zd\}$. 

\subsection{Weighted mixed norm spaces}\label{subsec:Lpq}
From now on  $v$ is a continuous, positive, even, submultiplicative weight function on $\rdd$, that is, 
$ v(z_1+z_2)\leq v(z_1)v(z_2)$, for every $ z_1,z_2\in\rdd$. A weight $m$ on $\rdd$ is \emph{$v$-moderate} if $m(z_1+z_2)\lesssim v(z_1)m(z_2)$  for all $z_1,z_2\in\rdd$.

%Observe that since $v$ is even, positive and submultiplicative, it follows that $v(z)\geq1$ for all $z \in\rdd$.  
 We write $m\in \mathcal{M}_v(\rdd)$ if $m$ is a positive, continuous, even and $v$-moderate weight function on $\rdd$. Important examples are the polynomial weights
\begin{equation}\label{vs}
	v_s(z) =(1+|z|)^{s},\quad s\in\bR,\quad z\in\rdd.
\end{equation}
Two weights $m_1,m_2$ are equivalent if $m_1\asymp m_2$. For instance, $v_s(z)\asymp (1+|z|^2)^{s/2}$.\\

If $m\in\cM_v(\rdd)$, $0<p,q\leq\infty$ and $f:\rdd\to\bC$ measurable, we set 
\[
	\norm{f}_{L^{p,q}_m}:=\left(\int_{\rd}\left(\int_{\rd}|f(x,y)|^pm(x,y)^p dx\right)^{q/p}dy\right)^{1/q},
\]
with the obvious adjustments when $\max\{p,q\}=\infty$. The space of measurable functions $f$ having $\norm{f}_{L^{p,q}_m}<\infty$ is denoted by $L^{p,q}_m(\rdd)$. %If $m\in\cM_v(\rdd)$ and $1\leq p,q\leq\infty$, then $L^{p,q}_{m}(\rdd)\ast L^1_v(\rdd)\hookrightarrow L^{p,q}_m(\rdd)$.\\
Recall the following partial generalization of the results in \cite{Fuhr}, contained in  \cite[Proposition 2.1]{CGframes}:

\begin{proposition}\label{thmA12}
		(i) Consider $A,D\in GL(d,\bR)$, $B\in \bR^{d\times d}$ and $0<p,q\leq\infty$. Define the {\bf upper triangular} matrix
		\begin{equation}\label{uppertr}
			S=\begin{pmatrix}
				A & B\\
				0_{d\times d} & D
			\end{pmatrix}.
		\end{equation}
		The mapping $\mathfrak{T}_S:f\in L^{p,q}(\rdd)\mapsto |\det(S)|^{1/2}f\circ S$  is an isomorphism of $L^{p,q}(\rdd)$ with bounded inverse $\mathfrak{T}_{S^{-1}}$.\\
		(ii) Consider $m\in\mathcal{M}_v(\rdd)$, $S\in GL(2d,\bR)$ and $0<p,q\leq\infty$, and the  dilation operator $(\mathfrak{T}_S)_m: f\in L^{p,q}_m(\rdd)\mapsto |\det(S)|^{1/2}f\circ S.$ If $m\circ S\asymp m$, then
	$\mathfrak{T}_S:L^{p,q}(\rdd)\to L^{p,q}(\rdd)$ is bounded if and only if $(\mathfrak{T}_S)_m:L^{p,q}_m(\rdd)\to L^{p,q}_m(\rdd)$ is bounded.
	\end{proposition}

\subsection{Time-frequency analysis tools}\label{subsec:23}
In this work, the Fourier transform of $f\in \cS(\rd)$ is normalized as
\[
\cF f=\hat f(\xi)=\int_{\rd} f(x)e^{-2\pi i\xi x}dx, \qquad \xi\in\rd.
\]
If $f\in\cS'(\rd)$, the Fourier transform of $f$ is defined by duality as the tempered distribution characterized by
\[
\langle \hat f,\hat\varphi\rangle=\langle f,\varphi\rangle, \qquad \varphi\in\cS(\rd).
\]
The operator $\cF$  is a surjective automorphism of $\cS(\rd)$ and $\cS'(\rd)$, as well as a surjective isometry of $L^2(\rd)$.
If $f\in\cS'(\rdd)$, we set $\cF_2f$, the partial Fourier transform with respect to the second variables: $$\cF_2(f\otimes g)=f\otimes\hat g,\quad f,g\in\cS'(\rd).$$
The \textit{short-time Fourier transform} of $f\in L^2(\rd)$ with respect to the window $g\in L^2(\rd)$ is defined in \eqref{STFTp}.

In information processing $\tau$-Wigner distributions  ($\tau\in\bR$) play a crucial role \cite{ZJQ21}. They are defined in \eqref{tauWigner}.
We recall the special cases $\tau=0$ and $\tau=1$, which  are the so-called (cross-)\textit{Rihacek distribution}
\begin{equation}\label{RD}
W_0(f,g)(x,\xi)=f(x)\overline{\hat g(\xi)}e^{-2\pi i\xi x}, \quad \phas\in\rd,
\end{equation}
 and (cross-)\textit{conjugate Rihacek distribution}
 \begin{equation}\label{CRD}
 W_1(f,g)(x,\xi)=\hat f(\xi)\overline{g(x)}e^{2\pi i\xi x}, \quad \phas\in\rd.
 \end{equation}
 %Observe that $W_\tau f(x,\xi)=\cF_2\mathfrak{T}_{L_\tau}(f\otimes\bar g)$, where for any $F$ on $\rdd$, $$\mathfrak{T}_{L_\tau}F(x,y)=F(x+\tau y,x-(1-\tau)y),\quad x,y\in\rd.$$

\subsection{Modulation  spaces \cite{KB2020,F1,Feichtinger_1981_Banach,book,Galperin2004,Kobayashi2006,PILIPOVIC2004194}} \label{subsec:MSs}
For $0<p,q\leq\infty$, $m\in\mathcal{M}_v(\rdd)$, and $g\in\cS(\rd)\setminus\{0\}$, the \textit{modulation space} $M^{p,q}_m(\rd)$ is defined as the space of tempered distributions $f\in\cS'(\rd)$ such that $$\norm{f}_{M^{p,q}_m}:=\Vert V_gf\Vert_{L^{p,q}_m}<\infty.$$ If $\min\{p,q\}\geq1$, the quantity $\norm{\cdot}_{M^{p,q}_m}$ is a norm, otherwise a quasi-norm. Different windows give  equivalent (quasi-)norms. Modulation spaces are (quasi-)Banach spaces, enjoying the inclusion properties: \\
if $0<p_1\leq p_2\leq\infty$, $0<q_1\leq q_2\leq\infty$, and $m_1,m_2\in\mathcal{M}_{v}(\rdd)$ satisfy $m_2\lesssim m_1$: $$ \cS(\rd)\hookrightarrow M^{p_1,q_1}_{m_1}(\rd)\hookrightarrow M^{p_2,q_2}_{m_2}(\rd)\hookrightarrow\cS'(\rd).$$ In particular, $M^1_v(\rd)\hookrightarrow M^{p,q}_m(\rd)$ for $m\in\mathcal{M}_v(\rdd)$ and $\min\{p,q\}\geq1$. %We will also use the inclusion $M^1_{m\otimes 1}(\rdd)\hookrightarrow L^1_m(\rdd)$. 
%We denote with $\cM^{p,q}_m(\rd)$ the closure of $\cS(\rd)$ in $M^{p,q}_m(\rd)$, which coincides with the latter when $p,q<\infty$. 
If $1\leq p,q<\infty$, $(M^{p,q}_m(\rd))'=M^{p',q'}_{1/m}(\rd)$, where $p'$ and $q'$ denote the Lebesgue dual exponents of $p$ and $q$, respectively. If $m_1\asymp m_2$, then $M^{p,q}_{m_1}(\rd)=M^{p,q}_{m_2}(\rd)$ for all $p,q$.

\subsection{The symplectic group $Sp(d,\mathbb{R})$ and the metaplectic operators}\label{subsec:26}
	A matrix $A\in\bR^{2d\times 2d}$ is symplectic, write $A\in Sp(d,\bR)$, if 
	\begin{equation}\label{fundIdSymp}
	A^TJA=J,\end{equation} where $J$ is the standard symplectic matrix defined in \eqref{defJ}.

	For $E\in GL(d,\bR)$ and $C\in Sym(2d,\bR)$, define:
	\begin{equation}\label{defDLVC}
		\cD_E:=\begin{pmatrix}
			E^{-1} & 0_{d\times d}\\
			0_{d\times d} & E^T
		\end{pmatrix} \qquad \text{and} \qquad V_C:=\begin{pmatrix}
			I_{d\times d} & 0\\ C & I_{d\times d}
		\end{pmatrix}.
	\end{equation}
	The matrices $J$, $V_C$ ($C$ symmetric), and $\cD_E$ ($E$ invertible) generate the group $Sp(d,\bR)$.\\

Recall  the Schr\"odinger representation $\rho$  of the Heisenberg group: $$\rho(x,\xi;\tau)=e^{2\pi i\tau}e^{-\pi i\xi x}\pi(x,\xi),$$ for all $x,\xi\in\rd$, $\tau\in\bR$. We will use the  property:  for all $f,g\in L^2(\rd)$, $z=(z_1,z_2),w=(w_1,w_2)\in\rdd$,
\[
	\rho(z;\tau)f\otimes\rho(w;\tau)g=e^{2\pi i\tau}\rho(z_1,w_1,z_2,w_2;\tau)(f\otimes g).
\]
For every $A\in Sp(d,\bR)$, $\rho_A(x,\xi;\tau):=\rho(A (x,\xi);\tau)$ defines another representation of the Heisenberg group that is equivalent to $\rho$, i.e., there exists a unitary operator $\hat A:L^2(\rd)\to L^2(\rd)$ such that:
\begin{equation}\label{muAdef}
	\hat A\rho(x,\xi;\tau)\hat A^{-1}=\rho(A(x,\xi);\tau), \qquad  x,\xi\in\rd, \ \tau\in\bR.
\end{equation}
This operator is not unique: if $\hat A'$ is another unitary operator satisfying (\ref{muAdef}), then $\hat A'=c\hat A$, for some constant $c\in\bC$, $|c|=1$. The set $\{\hat A : A\in Sp(d,\bR)\}$ is a group under composition and it admits the metaplectic group, denoted by $Mp(d,\bR)$, as subgroup. It is a realization of the two-fold cover of $Sp(d,\bR)$ and the projection:
 \begin{equation}\label{piMp}
	\pi^{Mp}:Mp(d,\bR)\to Sp(d,\bR)
\end{equation} is a group homomorphism with kernel $\ker(\pi^{Mp})=\{-id_{{L^2}},id_{{L^2}}\}$.

Throughout this paper, if $\hat A\in Mp(d,\bR)$, the matrix $A$ will always be the unique symplectic matrix such that $\pi^{Mp}(\hat A)=A$.

Recall the following basic facts on metaplectic operators.
\begin{proposition}{\cite[Proposition 4.27]{folland89}}\label{Folland427}
	Every operator $\hat A\in Mp(d,\bR)$ maps $\cS(\rd)$ isomorphically to $\cS(\rd)$ and it extends to an isomorphism on $\cS'(\rd)$.
\end{proposition}

For $C\in\R^{d\times d}$, define: 
\begin{equation}\label{chirp}
	\Phi_C(t)=e^{\pi i t Ct},\quad t\in\rd.
\end{equation}
If  $C\in Sym(2d,\bR)\cap GL(2d,\bR)$, then we can compute explicitly its Fourier transform, that is: 
\begin{equation}\label{ft-chirp}
\widehat{\Phi_C}=|\det(C)|\,\Phi_{-C^{-1}}.
\end{equation}
In what follows we list the most important examples of metaplectic operators.
\begin{example}\label{es22} Consider the symplectic matrices $J$, $\cD_L$ and $V_C$  defined  in (\ref{defJ}) and (\ref{defDLVC}), respectively. Then,
	\begin{enumerate}
		\item[\it (i)] $\pi^{Mp}(\cF)=J$;
		\item[\it (ii)] if $\mathfrak{T}_E:=|\det(E)|^{1/2}\,f(E\cdot)$, then $\pi^{Mp}(\mathfrak{T}_E)=\cD_E$;
		\item[\it (iii)] if $\phi_C f=\Phi_C f$ (multiplication by chirp), then $\pi^{Mp}(\phi_C)=V_C$;
		\item[\it (iv)] if $\psi_C =\cF \Phi_{-C} \cF^{-1}$ (Fourier multiplier), then $\pi^{Mp}(\psi_C)f=V_C^T$;
		\item[\it (v)] if $\cF_2 $ is the Fourier transform with respect to the second variables, then $\pi^{Mp}(\cF_2)=\cA_{FT2}$, where $\cA_{FT2}\in Sp(2d,\bR)$ is the $4d\times4d$ matrix with block decomposition
		\begin{equation}\label{AFT2}
		\cA_{FT2}:=\begin{pmatrix}
			I_{d\times d} & 0_{d\times d} & 0_{d\times d} & 0_{d\times d}\\
			0_{d\times d} & 0_{d\times d} & 0_{d\times d} & I_{d\times d} \\
			0_{d\times d} & 0_{d\times d} & I_{d\times d} & 0_{d\times d}\\
			0_{d\times d} & -I_{d\times d} & 0_{d\times d} & 0_{d\times d}
		\end{pmatrix}.
		\end{equation}
	\end{enumerate}

\end{example}

We will often use the following \textit{lifting-type} result, proved in \cite[Theorem B1]{CGshiftinvertible}:

If 	$	G_\cA\in Sp(d,\bR)$ has block decomposition 
\begin{equation}\label{Ga}
	G_\cA=\begin{pmatrix}
		G_{\cA_{11}} &   G_{\cA_{12}}\\
		G_{\cA_{21}} &  G_{\cA_{22}}
	\end{pmatrix}
\end{equation}
then it is easy to show that the $4d\times 4d$ matrix
\begin{equation}\label{liftmatrix}
	\Lift(G_\cA)=\begin{pmatrix}
		I_{d\times d} & 0_{d\times d} & 0_{d\times d} & 0_{d\times d}\\
		0_{d\times d} & G_{\cA_{11}} & 0_{d\times d} &  G_{\cA_{12}}\\
		0_{d\times d} & 0_{d\times d} & I_{d\times d} & 0_{d\times d}\\
		0_{d\times d} &  G_{\cA_{21}} & 0_{d\times d} &  G_{\cA_{22}}
	\end{pmatrix}
\end{equation}
is  symplectic and $\widehat{\Lift(G_\cA)}(f\otimes g)=f\otimes\widehat{G_\cA}g$ for every $f,g\in L^2(\rd)$.

\subsection{Metaplectic Wigner distributions}
Let $\hat\cA\in Mp(2d,\bR)$. The \textbf{metaplectic Wigner distribution} associated to $\hat\cA$ is defined in \eqref{WApre}.
The most famous  time-frequency representations  are metaplectic Wigner distributions. Namely, the STFT can be represented as  $$V_gf=\hat A_{ST}(f\otimes\bar g)$$  where:% if $I_{d\times d}$ denotes the identity of $\bR^{d\times d}$ and $0_{d\times d}$ denotes the matrix of $\bR^{d\times d}$ with all zero entries,
\begin{equation}\label{AST}
	A_{ST}=\begin{pmatrix}
		I_{d\times d} & -I_{d\times d} & 0_{d\times d} & 0_{d\times d}\\
		0_{d\times d} & 0_{d\times d} & I_{d\times d} & I_{d\times d}\\
		0_{d\times d} & 0_{d\times d} & 0_{d\times d} & -I_{d\times d}\\
		-I_{d\times d} & 0_{d\times d} & 0_{d\times d} &0_{d\times d}
	\end{pmatrix}
\end{equation}
and the $\tau$-Wigner distribution defined in \eqref{tauWigner} can be recast as $W_\tau(f,g)=\hat A_\tau(f\otimes\bar g)$, with
\begin{equation}\label{Atau}
	A_\tau=\begin{pmatrix}
		(1-\tau)I_{d\times d} & \tau I_{d\times d} & 0_{d\times d} & 0_{d\times d}\\
		0_{d\times d} & 0_{d\times d} & \tau I_{d\times d} & -(1-\tau)I_{d\times d}\\
		0_{d\times d} & 0_{d\times d} & I_{d\times d} & I_{d\times d}\\
		-I_{d\times d} & I_{d\times d} & 0_{d\times d} & 0_{d\times d}
	\end{pmatrix}.
\end{equation}
We recall the following continuity properties.
\begin{proposition}\label{prop25}
	Let $W_\cA$ be a metaplectic Wigner distribution. Then,\\
	 $W_\cA:L^2(\rd)\times L^2(\rd) \to L^2(\rdd) $ is bounded.
	 The same result holds if we replace  $L^2$ by $\cS$ or $\cS'$.
\end{proposition}
Since metaplectic operators are unitary, for all $f_1,f_2,g_1,g_2\in L^2(\rd)$,
\begin{equation}\label{Moyal}
	\la W_\cA(f_1,f_2),W_\cA(g_1,g_2)\ra = \la f_1,g_1\ra \overline{\la f_2,g_2\ra}.
\end{equation}
For the $4d\times4d$ symplectic matrix with block decomposition (\ref{blockA})
we define four $2d\times2d$ sub-matrices as follows:
\begin{equation}\label{defEAFA}
	E_\cA=\begin{pmatrix}
		A_{11} & A_{13}\\
		A_{21} & A_{23}
	\end{pmatrix}, \quad F_\cA=\begin{pmatrix}
		A_{31} & A_{33}\\
		A_{41} & A_{43}
	\end{pmatrix},
\end{equation}
and
\begin{equation}\label{defeafa}
		\cE_\cA=\begin{pmatrix}
		A_{12} & A_{14}\\
		A_{22} & A_{24}
	\end{pmatrix}, \quad \cF_\cA=\begin{pmatrix}
		A_{32} & A_{34}\\
		A_{42} & A_{44}
	\end{pmatrix}.
\end{equation}

\begin{remark}\label{remK}
	If $\cA\in Sp(2d,\bR)$, we can highlight the submatrices $E_\cA$, $\cE_\cA$, $F_\cA$ and $\cF_\cA$ as:
	\[
		\cA=\begin{pmatrix}
		E_\cA & \cE_\cA\\
		F_\cA & \cF_\cA
		\end{pmatrix}\mathcal{K},
	\] 
	where
	\[	
		\mathcal{K}=\begin{pmatrix} 
		I_{d\times d}  & 0_{d\times d} & 0_{d\times d} & 0_{d\times d}\\
		0_{d\times d} & 0_{d\times d} & I_{d\times d}  & 0_{d\times d}\\
		0_{d\times d} & I_{d\times d}  & 0_{d\times d} & 0_{d\times d}\\
		0_{d\times d} & 0_{d\times d} & 0_{d\times d} & I_{d\times d}
		\end{pmatrix}.
	\]
	We observe that $\cK$ is not symplectic. 
\end{remark}

Using the property of $\cA$ symplectic matrix, one can infer (see \cite[Section 2]{CGframes}) that:
\begin{equation}\label{relInTermsOfSubm}
	\begin{cases}
		E_\cA^TF_\cA-F_\cA^TE_\cA=J,\\
		\cE_\cA^T\cF_\cA-\cF_\cA^T\cE_\cA=J,\\
		E_\cA^T\cF_\cA-F_\cA^T\cE_\cA=0_{d\times d}.
	\end{cases}
\end{equation}

Finally, let us introduce the matrix:
\begin{equation}\label{defL}
	L=\begin{pmatrix}
		0_{d\times d} & I_{d\times d}\\
		I_{d\times d} & 0_{d\times d}
	\end{pmatrix} 
\end{equation}
The relationship among the matrices above are detailed below (see \cite[Lemma 2.6]{CGframes}).
\begin{lemma}\label{lemma44}
Consider the sub-matrices $E_\cA,F_\cA,\cE_\cA,\cF_\cA$  defined  in (\ref{defEAFA}) and (\ref{defeafa}). Let $L$ be defined as in (\ref{defL}).\\ 
	(i) If $E_\cA\in GL(2d,\bR)$, then\\
	(i.1) $\cF_\cA=E_\cA^{-T}F_\cA^T\cE_\cA$;\\
	(i.2) $G_\cA:=LE_\cA^{-1}\cE_\cA$ is symplectic;\\
	(i.3) $\cE_\cA\in GL(2d,\bR)$ and $\det(\cE_\cA)=(-1)^d\det(E_\cA)$.\\
	(ii) If we assume $\cE_\cA\in GL(2d,\bR)$, then,\\
	(ii.1) $F_\cA = \cE_\cA^{-T}\cF_\cA^T E_\cA$;\\
	(ii.2)  $\mathfrak{G}_\cA=L\cE_\cA^{-1}E_\cA$ is symplectic;\\
	(ii.3) $E_\cA\in GL(2d,\bR)$ and $\det(E_\cA)=(-1)^d\det(\cE_\cA)$.\\
	In particular, $E_\cA$ is invertible if and only if $\cE_\cA$ is invertible.
\end{lemma}
The action of $W_\cA$ on the time-frequency shift $\pi(w)$ can be exhibited in detail, thanks to the sub-blocks above, as explained in what follows.
\begin{lemma}\label{commpiWA}
	Let $W_\cA$ be a metaplectic Wigner distribution with $\cA=\pi^{Mp}(\hat\cA)$ having block decomposition \eqref{blockA}. Then, for $w\in\rdd$, $f,g\in L^2(\rd)$, we have
	\[
	W_\cA(\pi(w)f,g)=\Phi_{-M_\cA}(w)\pi(E_\cA w,F_\cA w)W_\cA(f,g),
	\]
	where $M_\cA$ is the symmetric matrix
	\begin{equation}\label{defMA}
		M_\cA=\begin{pmatrix}
			A_{11}^TA_{31}+A_{21}^TA_{41} & A_{31}^TA_{13}+A_{41}^TA_{23}\\
			A_{13}^TA_{31}+A_{23}^TA_{41} & A_{13}^TA_{33}+A_{23}^TA_{43}
		\end{pmatrix}.
	\end{equation}
\end{lemma}

Recall that $W_\cA$ is shift-invertible if $E_\cA\in GL(2d,\bR)$. 
For the STFT we obtain:
	$$ E_{A_{ST}}=I_{2d\times2d}=\left(\begin{array}{cc}
	I_{d\times d} & 0_{d\times d}\\
	0_{d\times d} &I_{d\times d}
\end{array}\right),
$$
which is invertible, whereas for the  $\tau$-Wigner distributions we have:
$$E_\tau:=E_{{ A}_{\tau}}=\left(\begin{array}{cc}
	(1-\tau) I_{d\times d} & 0_{d\times d}\\
	0_{d\times d} & \tau I_{d\times d}
\end{array}\right),$$
which is invertible if and only if $\tau\not=0$ and $\tau\not=1$, the cases of the Rihacek and conjugate-Rihacek distributions, which do not define modulation spaces.\\

We established the notation to state the characterization of shift-invertible matrices. The proof is based on the properties enjoyed by symplectic matrices and requires many computations, detailed in \cite[Theorem 4.2 and Corollary 4.3]{CGframes}.

\begin{theorem}\label{GiuLaMaschera0}
	Let $W_\cA$ be a shift-invertible metaplectic Wigner distribution and $G_\cA=LE_\cA^{-1}\cE_\cA$ be the matrix of Lemma \ref{lemma44}, with  $L$  as in \eqref{defL}.
	Then,
	\[
	\cA=\cD_{E_\cA^{-1}}V_{M_\cA}V_L^T \Lift(G_\cA),
	\]
	where $\Lift(G_\cA)$ is defined in \eqref{liftmatrix}.\par
	In particular, $W_\cA$ is shift-invertible if and only if, up to a sign,
	\begin{equation}\label{WA-STFT}
		W_\cA(f,g)=\mathfrak{T}_{E_\cA^{-1}}\Phi_{M_\cA+L}V_{\widehat{\delta_\cA} g}f,\quad f,g\in L^2(\rd),
	\end{equation}
	where
	\begin{equation}\label{defSA}
		\widehat{\delta_\cA} g:=\cF\widehat{\overline{G_\cA}}g,
	\end{equation}
	and $\widehat{\overline{G_\cA}}$ is the metaplectic operator with   $\pi^{Mp}(\widehat{\overline{G_\cA}})=\overline{G_\cA}$ where, assuming $G_\cA$ with the block decomposition in \eqref{Ga},
		\[
\overline{G_\cA}:=\begin{pmatrix}
	G_{\cA_{11}} &  - G_{\cA_{12}}\\
-	G_{\cA_{21}} &  G_{\cA_{22}}
\end{pmatrix}
	\] 
	(it is the matrix  $G_\cA$ with the second diagonal multiplied by $-1$).
\end{theorem}

\section{The submanifold of shift-invertible symplectic matrices}

We denote with
\[
	Sp_{inv}(2d,\bR):=\{\cA\in Sp(2d,\bR) \ : \ E_\cA\in GL(2d,\bR)\}
\]
the set of shift-invertible symplectic matrices $4d\times4d$.

\begin{remark}
	$Sp_{inv}(2d,\bR)$ is not a subgroup of $Sp(2d,\bR)$. Indeed, $A_{ST}\in Sp_{inv}(2d,\bR)$, whereas $A_{ST}^3$ is not shift-invertible. However, it is an open subset of $Sp(2d,\bR)$, since the mapping $\cA\in Sp(2d,\bR)\mapsto \det(E_\cA)\in\bR$ is smooth. Consequently, $Sp_{inv}(2d,\bR)$ is a submanifold of $Sp(2d,\bR)$ of dimension $2d(4d+1)$. 
\end{remark}

Another way to read Theorem \ref{GiuLaMaschera} is that any $\cA\in Sp_{inv}(2d,\bR)$ can be written as:
\[
	\cA=\cD_{E_\cA^{-1}}V_{M_\cA}V_L^T\Lift(G_\cA),
\]
where $E_\cA\in GL(2d,\bR)$, $M_\cA\in Sym(2d,\bR)$ and $G_\cA\in Sp(d,\bR)$. Therefore, shift-invertible distributions are described by these three matrices, rather than the blocks of $\cA$. The possibility of defining a shift-invertible Wigner distribution with arbitrary $E_\cA$ and $G_\cA$ is fundamental for the study of metaplectic Gabor frames, cf \cite{CGframes}.\\

The four submatrices $E_\cA$, $F_\cA$, $\cE_\cA$ and $\cF_\cA$ in (\ref{defEAFA}, \ref{defeafa}) determine the three submatrices $M_\cA$, $G_\cA$ and $E_\cA$ in Lemma \ref{lemma44} and (\ref{defMA}), and vice versa. Namely, 
\begin{align*}
	& M_\cA = E_\cA^T F_\cA -\begin{pmatrix} 0_{d	\times d} & I_{d	\times d}\\ 0_{d	\times d} & 0_{d	\times d}\end{pmatrix},\\
	& G_\cA = LE_\cA^{-1}\cE_\cA,
\end{align*}
whereas, by the definitions of $G_\cA$ and $M_\cA$, and by Lemma \ref{lemma44} $(i)$, 
\begin{align*}
	& \cE_\cA = E_\cA L G_\cA,\\
	& F_\cA = E_\cA^{-T}\left(M_\cA+\begin{pmatrix} 0_{d	\times d} & I_{d	\times d}\\ 0_{d	\times d} & 0_{d	\times d}\end{pmatrix}\right),\\
	& \cF_\cA = E_\cA^{-T}\left(M_\cA+\begin{pmatrix} 0_{d	\times d} & 0_{d	\times d}\\ I_{d	\times d} & 0_{d	\times d}\end{pmatrix}\right)LG_\cA.
\end{align*}

In this section, we prove rigorously that every triple $E\in GL(2d,\bR)$, $C\in Sym(2d,\bR)$ and $S\in Sp(d,\bR)$ determines uniquely a symplectic matrix in $Sp_{inv}(2d,\bR)$.\\

To simplify the notation, let
\[
	CG(2d,\bR):=GL(2d,\bR)\times Sym(2d,\bR)\times Sp(d,\bR),
\]
and denote:
\begin{equation}\label{defcAmap}
	\alpha(E,C,S):=\cD_{E^{-1}}V_CV_L^T\Lift(S),
\end{equation}
where $L$ is defined as in \eqref{defL}.

\begin{theorem}
	The mapping $\alpha:(E,C,S)\mapsto \cA(E,C,S)$ is a set bijection from $CG(2d,\bR)$ to $Sp_{inv}(2d,\bR)$. 
\end{theorem}
\begin{proof}
	We first observe that $\cA=\alpha(E,C,S)\in Sp_{inv}(2d,\bR)$ for every $(E,C,S)\in CG(2d,\bR)$ with $E_{\cA}=E$. Indeed, a simple computation shows that:
	\begin{equation}\label{eq456}
		\cD_{E^{-1}}V_CV_L^T\Lift(S)=\begin{pmatrix} E & ELS\\ \ast & \ast \end{pmatrix}\mathcal{K},
	\end{equation}
	where $\cK$ is defined as in Remark \ref{remK}.
	This highlights that $E_\cA=E$ and $G_\cA=S$. In particular, $\cA\in Sp_{inv}(2d,\bR)$. 
	
	If $\cA\in Sp_{inv}(2d,\bR)$, it follows by Theorem \ref{GiuLaMaschera} that $\cA=\alpha(E_\cA,M_\cA,G_\cA)$, with $(E_\cA,M_\cA,G_\cA)\in CG(2d,\bR)$, so the surjectivity of $\alpha$ follows. To prove the injectivity, observe that if $\alpha(E_1,C_1,S_1)=\alpha(E_2,C_2,S_2)$, then
	\[
		\begin{pmatrix} E_1 & E_1LS_1\\ \ast & \ast \end{pmatrix}
		=
		\begin{pmatrix} E_2 & E_2LS_2\\ \ast & \ast \end{pmatrix}
	\]
	by (\ref{eq456}). This entails that if $\alpha(E_1,C_1,S_1)=\alpha(E_2,C_2,S_2)$, then $E_1=E_2$ and, consequently, $S_1=S_2$. It remains to prove that $C_1= C_2$. We have:
	\[
		\cD_{E_1^{-1}}V_{C_1}V_L^T\Lift(S_1)=\cD_{E_2^{-1}}V_{C_2}V_L^T\Lift(S_2)
	\]
	if and only if
	\[
		V_{C_1}=\cD_{E_1^{-1}}^{-1}\cD_{E_2^{-1}}V_{C_2}V_L^T\Lift(S_2)\Lift(S_1)^{-1}V_L^{-T}.
	\]
Since $E_1=E_2$ and $S_1=S_2$, we obtain
	\[
		V_{C_1}=V_{C_2}
	\]
	and, therefore, $C_1 = C_2$.
\end{proof}

\begin{corollary}
	A metaplectic Wigner distribution $W_{\cA}$ is shift-invertible if and only if 
	\begin{equation}\label{eq654}
		W_\cA(f,g)(z)=|\det(E)|^{-1}\Phi_{C}(E^{-1}z)V_{\hat \delta g}f(E^{-1}z), \qquad f,g\in L^2(\rd), \ z\in\rdd,
	\end{equation}
	for some $E\in GL(2d,\bR)$, $C\in Sym(2d,\bR)$ and $\hat\delta\in Mp(d,\bR)$. 
\end{corollary}
\begin{proof}
By (\ref{WA-STFT}), if $W_\cA$ is shift-invertible, then (\ref{eq654}) holds for the triple $E_\cA, M_\cA+L$ and $\widehat{\delta_\cA}$. It remains to check that (\ref{eq654}) defines a shift-invertible metaplectic Wigner distribution. Let $W_\cA$ be as in (\ref{eq654}) and $\cA$ be the related symplectic matrix. Then,
\begin{align*}
	\cA&=\cD_{E^{-1}}V_CA_{ST}\Lift(\bar\delta)=\cD_{E^{-1}}V_CV_{-L}V_L^T\cA_{FT2}\Lift(\bar\delta)\\
	&=\cD_{E^{-1}}V_{C-L}V_L^T\Lift(J\bar\delta)
\end{align*}
where $\overline\delta$ is the projection of the metaplectic operator:
\[
	\overline{\hat\delta g}=\hat{\bar \delta}{\bar g}.
\]

Thus, the shift-invertibility of $W_\cA$ follows by Theorem \ref{GiuLaMaschera}.
\end{proof}

In view of the issues above, the characterization of modulation spaces can be easily rephrased as follows.
\begin{theorem}\label{thm987}
	Consider $0<p,q\leq\infty$, $m\in\cM_v(\rdd)$ with $m\asymp m\circ E^{-1}$, $g\in \cS(\rd)\setminus\{0\}$. For any $f\in\cS'(\rd)$,  define 
	\[
		W_\cA(f,g)(z)=|\det(E)|^{-1}\Phi_{C}(E^{-1}z)V_{\hat \delta g}f(E^{-1}z), \qquad  z\in\rdd,
	\]
where $E\in GL(2d,\bR)$, $C\in Sym(2d,\bR)$ and $\hat\delta\in Mp(d,\bR)$. \\
	(i) If $E$ is upper-triangular, then $$\norm{f}_{M^{p,q}_m}\asymp \norm{W_\cA(f,g)}_{L^{p,q}_m}.$$\\
	(ii) If $p=q$, the upper triangularity assumption in (i) can be dropped.
\end{theorem}

It is not difficult to show similar results for Wiener amalgam spaces \cite{Feichtinger_1990_Generalized}. We leave the details to the interested reader.

\section*{Declarations}

{\bf Conflict of interest.}  All authors declare that they have no conflicts of interest.

\end{document}